\documentclass[sn-mathphys,Numbered]{sn-jnl}% Math and Physical Sciences Reference Style
\usepackage{graphicx}%
\usepackage{multirow}%
\usepackage{amsmath,amssymb,amsfonts}%
\usepackage{amsthm}%
\usepackage{mathrsfs}%
\usepackage[title]{appendix}%
\usepackage{xcolor}%
\usepackage{textcomp}%
\usepackage{manyfoot}%
\usepackage{booktabs}%
\usepackage{algorithm}%
\usepackage{algorithmicx}%
\usepackage{algpseudocode}%
\usepackage{listings}%
\theoremstyle{thmstyleone}%
\newtheorem{theorem}{Theorem}%  meant for continuous numbers
\theoremstyle{thmstyletwo}%
\newtheorem{remark}{Remark}%

\theoremstyle{thmstylethree}%
\newtheorem{definition}{Definition}%
\newtheorem{corollary}{Corollary}

\begin{document}

\title[Article Title]{Fractional Calculus Operator Emerging from the 2D Biorthogonal Hermite Konhauser Polynomials}

\author[1]{\fnm{Mehmet Ali} \sur{Özarslan}}\email{mehmetali.ozarslan@emu.edu.tr} 

\author*[1]{\fnm{İlkay} \sur{Onbaşı Elidemir}}\email{ilkay.onbasi@emu.edu.tr}

\equalcont{These authors contributed equally to this work.}

\affil[1]{Department of Mathematics, Faculty of Arts and Sciences, Eastern Mediterranean University, Gazimagusa, TRNC, Mersin 10, Turkey}

\abstract{In this paper, we present a novel method for constructing families of two-variable biorthogonal polynomials by utilizing known families of one-variable biorthogonal and orthogonal polynomials. Using this approach, we define a new class of two-dimensional Hermite-Konhauser (H-K) polynomials and investigate several of their fundamental properties, including biorthogonality, operational formulas, and integral representations. We also examine the behavior of these polynomials under Laplace transformations and the actions of fractional integral and derivative operators. Corresponding to the introduced polynomials, we define a new class of bivariate H-K Mittag-Leffler functions and derive their analogous properties. To extend the framework of fractional calculus, we introduce two additional parameters and study the modified two-dimensional Hermite-Konhauser polynomials and the corresponding bivariate H-K Mittag-Leffler functions. Additionally, we propose an integral operator whose kernel involves the modified bivariate H-K Mittag-Leffler function. We demonstrate that this operator satisfies the semigroup property and establish its left inverse, which corresponds to a fractional derivative operator.}




\maketitle

\section{Introduction}
 Two-dimensional Laguerre-Konhauser polynomials are defined by the formula \cite{1}
\begin{equation*}
_{\Upsilon}L_n^{(\varkappa,\varrho)}(X,Y)=n!\sum_{s=0}^n \sum_{r=0}^{n-s}\frac{(-1)^{s+r}X^{r+\varkappa}Y^{\Upsilon s+\varrho}}{s!r!(n-s-r)!\Gamma(\varkappa+r+1)\Gamma(\Upsilon s+\varrho+1)}  
\end{equation*}
$ (\varkappa,\varrho, \Upsilon=1,2,...).$ \\
Note that 
\begin{equation*}
_{\Upsilon}L_n^{(\varkappa,\varrho)}(X,Y)=n!\sum_{s=0}^n\frac{(-1)^s X^{s+\varkappa}Y^\varrho Z_{n-s}^{\varrho}(Y;\Upsilon)}{s!\Gamma(\varkappa+s+1)\Gamma(\Upsilon n-\Upsilon s+\varrho+1)}
\end{equation*}
and
\begin{equation*}
_{\Upsilon}L_n^{(\varkappa,\varrho)}(X,Y)=n!\sum_{s=0}^n\frac{(-1)^s X^{\varkappa}Y^{\Upsilon s+\varrho}L_{n-s}^\varkappa(X)}{s!\Gamma(\varkappa+n-s+1)\Gamma(\Upsilon s+\varrho+1)},
\end{equation*}
where $L_n^\varkappa(X)$ are the Laguerre polynomials and $Z_n^\varrho(Y;\Upsilon)$ are the Konhauser polynomials.
\par
The polynomials $Z_n^{\varrho}(t;\Upsilon)$ and $Y_n^{\varrho}(t;\Upsilon)$ ($\varkappa>-1$ and $\Upsilon=1,2,,\dots$) are  the pair of biorthogonal sets with respect to a weight function $w(t)=t^\varkappa e^{-t}$  over the interval $(0,\infty)$ that are suggested by the Laguerre polynomials where, ( see \cite{carl}, \cite{kon}, \cite{jde}, \cite{hc})
\begin{equation*}
Z_n^{\varrho}(t;\Upsilon)=\frac{\Gamma(1+\varrho+\Upsilon n)}{n!}\sum_{r=0}^n\frac{(-n)_rt^{\Upsilon r}}{r!\Gamma(\varrho+1+\Upsilon r)},
\end{equation*}
\begin{equation*}
Y_n^{(\varrho)}(t;\Upsilon)=\frac{1}{n!}\sum_{i=0}^n\frac{t^i}{i!}\sum_{j=0}^i (-1)^j\binom{i}{j}\bigg(\frac{j+\varrho+1}{\Upsilon}\bigg)_n.
\end{equation*}

In \cite{cemo} Özarslan and Kürt introduced the second set $_{\Upsilon}\mathfrak{L}_n^{(\varkappa,\varrho)}(X,Y)$, by
\begin{equation*}
_{\Upsilon}\mathfrak{L}_n^{(\varkappa,\varrho)}(X,Y)=L_n^\varkappa(X)\sum_{s=0}^{max\{m,n\}} Y_s^{(\varrho)}(Y;\Upsilon),
\end{equation*}

and  using the orthogonality and biorthogonality relations,
\begin{equation*}
\int_{0}^{\infty} e^{-X}X^{\varkappa}L_n^\varkappa(X)L_m^\varkappa(X)dX=\frac{\Gamma(n+\varkappa+1)}{n!}\delta_{nm}
\end{equation*}
and
\begin{equation*}
\int_{0}^{\infty} e^{-X}X^{\varrho}Z_n^\varrho(X;\Upsilon)Y_m^\varrho(X;\Upsilon)dX=\frac{\Gamma(\Upsilon n+\varrho+1)}{n!}\delta_{nm}
\end{equation*}
they have shown that, the sets $_{\Upsilon}\mathfrak{L}_n^{(\varkappa,\varrho)}(\bar{X},\bar{Y})$ and $_{\Upsilon}L_n^{(\varkappa,\varrho)}(\bar{X},\bar{Y})$ are  bi-orthonormal with  $w(\bar{X},\bar{Y})=e^{-(\bar{X}+\bar{Y})}$  on $(0,\infty) \times (0,\infty)$. More precisely,
\begin{equation*}
\int_{0}^{\infty}\int_{0}^{\infty}e^{-\bar{X}-\bar{Y}}{_{\Upsilon}\mathfrak{L}_m^{(\varkappa,\varrho)}(\bar{X},\bar{Y})}_{\Upsilon}L_n^{(\varkappa,\varrho)}(\bar{X},\bar{Y})d\bar{Y}d\bar{X}=\delta_{nm},
\end{equation*}
where $\delta_{nm}$ is the Kronecker's delta.
\\
Orthogonal polynomials in two variables are studied as the natural generalization of orthogonal polynomials in one variable. In 1967, Krall and Sheffer extended the concept of classical sequence of orthogonal polynomials to the bivariate case (see \cite{orth}).
In 1975, Koornwinder introduced several examples of two-variable analogues of Jacobi polynomials, presenting seven classes of orthogonal polynomials that serve as bivariate analogues (see \cite{koor}).These examples include classical polynomials as defined by Krall and Sheffer. For further insights, see the book 'Orthogonal Polynomials in Two Variables'\cite{suetin}. 
\\
Hermite polynomials $H_n(x)$ are a sequence of orthogonal polynomials that arise in various areas of mathematics and physics, especially in solving differential equations and in quantum mechanics.Complex Hermite polynomials generalize Hermite polynomials to the complex domain. Their properties and applications are extensive, ranging from probability theory to signal processing. In \cite{ref1}, the authors explore the intricate relationships between complex Hermite polynomials, Lie algebra representations, and operational techniques. In \cite{ref2}, a novel class of two-variable complex Hermite polynomials is introduced, characterized by quasi-monomiality and derived using Burchnall-type operational identities and the monomiality principle.
\\
Corresponding to the polynomials $_{\Upsilon}L_n^{(\varkappa,\varrho)}(X,Y)$, Özarslan and Kürt \cite{cemo} introduced a bivariate Mittag  Leffler functions $E^{(\gamma)}_{\varkappa,\varrho,\Upsilon}(X,Y)$  by
\begin{equation}
E_{\varkappa,\varrho,\Upsilon}^{(\gamma)}(X,Y)=\sum_{r=0}^\infty\sum_{s=0}^\infty \frac{(\gamma)_{r+s}X^rY^{\Upsilon s}}{r!s!\Gamma(\varkappa+r)\Gamma(\varrho+\Upsilon s)},    \label{cemo}
\end{equation}
$(\varkappa, \varrho, \gamma \in\mathbb{C}, Re(\varkappa), Re(\varrho), Re(\Upsilon)>0), $ where
\begin{equation*}
_{\phantom{1}\Upsilon}L_n^{(\varkappa,\varrho)}(X,Y)=X^{\varkappa} Y^{\varrho} E_{\varkappa+1,\varrho+1,\Upsilon}^{(-n)}(X,Y).
\end{equation*}

$E^{(\gamma)}_{\varkappa,\varrho,\Upsilon}(X,Y)$ is function of two variables $X$ and $Y$, expressed as double power series which may be seen as analogous to the following power series for the classical Mittag-Leffler function,
\begin{equation*}
E_{\varkappa} (X) =\sum_{s=0}^\infty \frac{X^s}{\Gamma(1+\varkappa s)}   \quad  Re(\varkappa)>0.
\end{equation*}
In 1903 Gösta Mittag-Leffler defined the classical Mittag-Leffler function. Many generalisations exist for the function $E_{\varkappa} (X)$ (see \cite{book}, \cite{luchko}, \cite{luchko2}). Some generalised forms consist two or more  variables instead of single variable. Mittag-Leffler function is the important type of the special functions, in connection with  fractional calculus. After the work by Prabhakar \cite{prab}, these functions appeared in the kernel of many fractional calculus operators \cite{kilbas}.
 In recent years bivariate versions of these functions have shown to have interesting applications in applied science \cite{natur}. There has been significant interest in the investigation of different variant of bivariate Mittag-Leffler functions and fractional calculus operators having these functions in the kernel especially in the recent years. (see \cite{ref}, \cite{ng}, \cite{garg}, \cite{kilbas2}, \cite{kilbas3}, \cite{shukla}, \cite{aa}, \cite{rk}, \cite{ma}, \cite{ma1}, \cite{ck}, \cite{kc}, \cite{arren}, \cite{new}, \cite{new2}, \cite{new3}.)

 Motivated by the results in \cite{cemo}, we introduce a new method for constructing bivariate biorthogonal polynomial families using known one-variable biorthogonal and orthogonal polynomials.
 
In section 3, we define the bivariate biorthogonal polynomials namely,  2D - Hermite Konhauser polynomials, denoted by $_{\Upsilon}H_n^{\varrho}(X,Y)$, derived from Hermite and Konhauser polynomials. We then define the corresponding bivariate  H - K  Mittag-Leffler functions,  $E_{\varrho,\Upsilon}^{(\gamma_1;\gamma_2; \gamma_3)}(X,Y)$ and study various properties, including their operational and integral representations, Laplace transforms, and behavior under fractional integrals and derivatives.

In section 4,  we introduce two additional parameters to define the modified  2D - Hermite Konhauser polynomials, $_{\Upsilon}H_n^{\varkappa,\varrho; c}(X,Y)$, and corresponding bivariate H - K Mittag-Leffler function, $E_{\varkappa,\varrho,\Upsilon}^{(\gamma_1;\gamma_2; \gamma_3,\gamma_4)}(X,Y)$. We investigate their properties and solve a double integral equation involving  $_{\Upsilon}H_n^{\varkappa,\varrho; c}(X,Y)$  expressing the solution in terms of $E_{\varkappa,\varrho,\Upsilon}^{(\gamma_1;\gamma_2; \gamma_3,\gamma_4)}(X,Y)$.This leads to a new fractional integral operator containing $E_{\varkappa,\varrho,\Upsilon}^{(\gamma_1;\gamma_2; \gamma_3,\gamma_4)}(X,Y)$. We analyze its transformation properties in Lebesgue summable function spaces and derive the corresponding fractional derivative as its left inverse.

\section{General Method to Construct Bivariate Orthogonal polynomials}

This section aims to introduce a general method for constructing bivariate biorthogonal polynomials. The principal result is stated in the following theorem.

\begin{theorem}
Let $w(X)$ be an admissible weight function over $(A,B)$, let $R_n(X)$ and $S_n(X)$ be biorthogonal polynomials of the basic polynomials $r(X)$ and $s(X)$, respectively. In  other words, they satisfy the biorthogonality condition
\begin{equation*}
\int_A^B w(X)R_m(X)S_n(X)dX=J_{n,m}:=\begin{cases}
      0 ,& \text{if} \quad  n\neq m\\
      J_{n,n} , & \text{if}  \quad  n=m
    \end{cases}.
\end{equation*}
Furthermore suppose, 
\begin{equation*}
K_n(X)=\sum_{i=0}^n D_{n,i}[r(X)]^i
\end{equation*}
and
\begin{equation*}
\int_C^D w(X)K_m(X)K_n(X)dX=||K_n||^2\delta_{nm}
\end{equation*}
where $\delta_{nm}$ is the Kronecker delta.
Then the bivariate polynomials 
\begin{equation}
P_n(X,Y)=\sum_{s=0}^n \frac{D_{n,s}}{J_{n-s,n-s}}[r(X)]^s R_{n-s}(Y) \quad \text{and} \quad Q_n(X,Y)=K_n(X)\sum_{j=0}^{max\{m,n\}}  S_j(Y)\label{1}
\end{equation}
form a biorthogonal pair with respect to the weight function $w(X)w(Y)$ over the domain $(A,B)\times(C,D)$.
\end{theorem}

\begin{proof}
Without loss of generality, let $m>n$. Then
\begin{equation*}
\begin{aligned}
{}&\int_C^D \int_A^B w(X)w(Y)P_n(X,Y)Q_m(X,Y)dYdX \\&
=\int_C^D w(X)K_m(X)\sum_{s=0}^n \frac{D_{n,s}}{J_{n-s,n-s}}[r(X)]^sdX \\& \times  \int_A^B w(Y)R_{n-s}(Y)\sum_{j=0}^m S_j(Y)dY. \\ &
=\int_C^D w(X)K_m(X)\sum_{s=0}^n \frac{D_{n,s}}{J_{n-s,n-s}}[r(X)]^sdXJ_{n-s,n-s} \\&
=\int_C^D w(X)K_m(X)K_n(X)dX=||K_n||^2\delta_{nm}.
\end{aligned}
\end{equation*}
Using the biorthogonality condition, the second integral evaluates to $J_{n-s,n-s} $ yielding
\begin{equation*}
\int_C^D w(X)K_m(X)K_n(X)dX=||K_n||^2\delta_{nm}.
\end{equation*}
\end{proof}

\begin{remark}
Applying the above theorem, we conclude that the 2D Laguerre Konhauser polynomials, $_{\Upsilon}L_n^{(\varkappa,\varrho)}(X,Y)$, and polynomials,  $_{\Upsilon}\mathfrak{L}_n^{(\varkappa,\varrho)}(X,Y)$,constitute a pair of bivariate biorthogonal polynomials with respect to the weight function $w(X,Y)=e^{-X-Y}$ over the domain $(0,\infty) \times (0,\infty)$.
\end{remark}

\section{ The  2D - Hermite Konhauser polynomials}

In this section, we apply the method introduced earlier to define the two-dimensional Hermite-Konhauser polynomials using univariate Hermite and Konhauser polynomials.

Recall that Hermite polynomials $H_n(X)$  are given by
\begin{equation}
H_n(X)=\sum_{s=0}^{[\frac{n}{2}]}\frac{(-1)^s(-n)_{2s}}{s!}(2X)^{n-2s},
\end{equation}  \label{2}
and they satisfy the following orthogonality condition,
\begin{equation*}
\int_{-\infty}^{\infty} e^{-X^2}H_n(X)H_m(X)dX=\begin{cases}
      0 ,& \text{if} \quad  n\neq m\\
      2^{n}n!\sqrt{\pi} & \text{if}  \quad  n=m
    \end{cases}.
\end{equation*}
In the following definition, we define 2D Hermite Konhauser polynomials by choosing  $r(X)=2X$ , $R_n(X)=Z_n^{\varrho}(X;\Upsilon) $, $S_n(X)=Y_n^{(\varrho)}(X;\Upsilon)$ and $K_n(X)=H_n(X)$ in equation \eqref{1}.

\begin{definition}
The 2D - Hermite Konhauser polynomials are defined as follows;
\begin{equation}
_{\Upsilon}H_n^{\varrho}(X,Y)=\sum_{s=0}^{[\frac{n}{2}]}\sum_{r=0}^{n-s}\frac{(-1)^s(-n)_{2s}(-n)_{s+r}}{(-n)_s\Gamma(\varrho+1+\Upsilon r)s!r!}(2X)^{n-2s}Y^{\Upsilon r} \label{3}
\end{equation}
where  $\varrho>-1$ and $\Upsilon=1,2\dots$.
\end{definition}

\begin{theorem}
The 2D - Hermite Konhauser polynomials, ${}_{\phantom{1}\Upsilon}H_n^{\varrho}(X,Y)$, can be written in the following form:
\begin{equation}
_{\Upsilon}H_n^{\varrho}(X,Y)=\sum_{s=0}^{[\frac{n}{2}]}\frac{(-1)^s(-n)_{2s}(n-s)!}{\Gamma(1+\varrho+\Upsilon n-\Upsilon s)s!}(2X)^{n-2s}Z_{n-s}^\varrho(Y;\Upsilon). \label{4}
\end{equation}
\end{theorem}
\begin{proof}
This result follows directly from the representation formulas for ${}_{\phantom{1}\Upsilon}H_n^{\varrho}(X,Y)$ and $Z_{n}^\varrho(X;\Upsilon)$.
\end{proof}

\begin{theorem}
The 2D - Hermite Konhauser polynomials satisfy the following biorthogonality condition;
\begin{equation}
\int_0^{\infty }\int_{-\infty}^{\infty}e^{-X^2-Y}  Y^\varrho_{\phantom{1}\Upsilon}H_n^{\varrho}(X,Y)Q_m(X,Y)dXdY=\begin{cases}
      0 ,& \text{if} \quad  n\neq m\\
      2^{n}n!\sqrt{\pi} & \text{if}  \quad  n=m
    \end{cases},  \label{5}
\end{equation}
where
\begin{equation}
Q_m(X,Y)=H_m(X)\sum_{j=0}^{max\{m,n\}}  Y_j^{(\varrho)}(Y;k). \label{6}  
\end{equation}
\end{theorem}

\begin{proof}
Without loss of generality let $m>n$. Replacing $_{\Upsilon}H_n^{\varrho}(X,Y)$  with \eqref{4} and $Q_m(X,Y)$  with \eqref{6} in the left side we get
\begin{equation*}
\begin{aligned}
{}&\int_0^{\infty }\int_{-\infty}^{\infty}e^{-X^2-Y}  Y^\varrho \sum_{s=0}^{[\frac{n}{2}]}\frac{(-1)^s(-n)_{2s}(n-s)!}{\Gamma(1+\varrho+\Upsilon n-\Upsilon s)s!}(2X)^{n-2s}Z_{n-s}^\varrho(Y;\Upsilon) H_m(X) \\& \times \sum_{j=0}^m Y_j^{(\varrho)}(Y;k)dXdY 
=\int_{-\infty}^{\infty}  e^{-X^2} \sum_{s=0}^{[\frac{n}{2}]}\frac{(-1)^s(-n)_{2s}}{s!}(2X)^{n-2s}H_m(X)dX \\&
=\int_{-\infty}^{\infty}  e^{-X^2}H_n(X)H_m(X)dX.
\end{aligned} 
\end{equation*}
Whence, we obtain the result.
\end{proof}

From the view of the definition of the Kampe de Feriet's double hypergeometric series: \cite{10}
\begin{equation*}
F_{k,l,m}^{n,p,q}\begin{bmatrix}( a_n):(b_p);(c_q);&\\ & X;Y \\ (d_k) :(f_l);(g_m); & \end{bmatrix}=\sum_{r,s=0}^\infty \frac{\prod_{i=1}^n (a_i)_{r+s}\prod_{i=1}^p(b_i)_{r}\prod_{i=1}^q (c_i)_{s}X^sY^r}{\prod_{i=1}^k (d_i)_{r+s}\prod_{i=1}^l (f_i)_{r}\prod_{i=1}^m (g_i)_{s}s!r!},
\end{equation*}
we derive the following series representation for the polynomials $_{\Upsilon}H_n^{\varrho}(X,Y)$.

\begin{theorem} \label{thm4}
The  2D Hermite Konhauser polynomials can expressed in terms of Kampe de Feriet's hypergeometric functions by
\begin{equation}
{}_{\phantom{1}\Upsilon}H_n^{\varrho}(X,Y)=\frac{(2X)^n}{\Gamma(1+\varrho)}F_{0,\Upsilon,1}^{1,0,2}\begin{bmatrix} -n:-;\bigtriangleup(2;-n);&\\ & \frac{-1}{X^2};\bigg(\frac{Y}{\Upsilon}\bigg)^\Upsilon \\ - :\bigtriangleup(\Upsilon;\varrho+1);(-n); & \end{bmatrix},
\end{equation}
where $\bigtriangleup(k;\sigma)$ denotes the $k$ prameters $\frac{\sigma}{k},\frac{\sigma+1}{k},\dots,\frac{\sigma+k-1}{k}$.
\end{theorem}
\begin{proof}
Multiplying \eqref{3} by$\frac{\Gamma(1+\varrho)}{\Gamma(1+\varrho)}$ and using the property of pochammer symbol $[(a)_{2s}=a^{2s}(\frac{a}{2})_s(\frac{a+1}{2})_s]$, we get
\begin{equation}
{}_{\phantom{1}\Upsilon}H_n^{\varrho}(X,Y)=\frac{(2X)^n}{\Gamma(1+\varrho)}\sum_{s=0}^{[\frac{n}{2}]}\sum_{r=0}^{n-s}\frac{(-1)^s(-n)_{s+r}(\frac{-n}{2})_s(\frac{-n+1}{2})_s}{s!r!(-n)_s(\varrho+1)_{\Upsilon r}}\left(\frac{-1}{X^2}\right)^sY^{\Upsilon r},
\end{equation}
where $(1+\varrho)_{\Upsilon r}=\Upsilon^{\Upsilon r}\prod_{j=0}^{\Upsilon-1}\bigg(\frac{\varrho+1+j}{\Upsilon}\bigg)_r$.
\\
Whence  the result.
\end{proof}

Next, we introduce a new bivariate H - K Mittag-Leffler function $E_{\varrho,\Upsilon}^{(\gamma_1;\gamma_2; \gamma_3)}(X,Y)$ and establish its relationship with the 2D Hermite-Konhauser polynomials.

\begin{definition}
Bivariate H - K Mittag-Leffler function $E_{\varrho,\Upsilon}^{(\gamma_1;\gamma_2;\gamma_3)}(X,Y)$  is defined by the formula,
\begin{equation}
E_{\varrho,\Upsilon}^{(\gamma_1;\gamma_2;\gamma_3)}(X,Y)=\sum_{s=0}^\infty\sum_{r=0}^\infty\frac{(\gamma_1)_{2s}(\gamma_2)_{s+r}X^sY^{\Upsilon r}}{(\gamma_3)_s\Gamma(\varrho+\Upsilon r)r!s!},    \label{9}
\end{equation}
$(\Upsilon, \varrho, \gamma_1,\gamma_2, \gamma_3 \in\mathbb{C}, Re(\gamma_1)>0,  Re(\gamma_2)>0,  Re(\gamma_3)>0,  Re(\varrho)>0, Re(\Upsilon)>0).$
\end{definition}

By setting $ X=0$, this function reduces to the classical form
\begin{equation*}
E_{\varrho,\Upsilon}^{(\gamma_1;\gamma_2;\gamma_3)}(0,Y)=E_{\Upsilon,\varrho}^{\gamma_2}(Y^\Upsilon)
\end{equation*}
where $E_{\varkappa,\varrho}^{\gamma}(z)$ is the Prabhakar Mittag-Leffler function defined by  \cite{prab}
\begin{equation*}
E_{\varkappa,\varrho}^{\gamma}(z)=\sum_{n=0}^{\infty}\frac{(\gamma)_{n}z^n}{\Gamma(\varkappa n+\varrho)n!},
\end{equation*}
$ (\varkappa, \varrho, \gamma \in\mathbb{C}, Re(\varkappa)>0, Re(\varrho)>0 ,Re(\gamma)>0).$ 

\begin{corollary}
By comparing \eqref{3} and \eqref{9} it is easily seen that;
\begin{equation}
{}_{\phantom{1}\Upsilon}H_n^{\varrho}(X,Y)=(2X)^nE_{\varrho+1,\Upsilon}^{(-n;-n;-n)}\bigg(\frac{-1}{4X^2},Y\bigg). \label{10}
\end{equation}
\end{corollary}

\subsection{Operational and Integral  Representation for the  Bivariate H - K Mittag-Leffler Function and  2D - Hermite Konhauser polynomials}

In this subsection,  we first present operational  representations for the  bivariate H - K Mittag-Leffler function $E_{\varrho,\Upsilon}^{(\gamma_1;\gamma_2;\gamma_3)}(X,Y)$ and for 2D - Hermite Konhauser polynomials ${}_{\phantom{1}\Upsilon}H_n^{\varrho}(X,Y)$. We also express the 2D - Hermite Konhauser polynomials in terms of Bessel polynomials. Additionally, we derive integral representations for both $E_{\varkappa,\varrho,\Upsilon}^{(\gamma_1;\gamma_2;\gamma_3)}(X,Y)$ and ${}_{\phantom{1}\Upsilon}H_n^{\varrho}(X,Y)$.\\
We recall the inverse derivative operator and generalized hypergeometric functions respectively
\begin{equation*}
D_X^{-1}f(X)=\int_0^X f(t)dt,
\end{equation*}
\begin{equation*}
{}_{\phantom{1}m}F_n(a_1, \dots, a_m; b_1,\dots,b_n ; X)=\sum_{k=0}^\infty\frac{(a_1)_k\dots(a_m)_kX^k}{(b_1)_k\dots(b_n)_kk!}.
\end{equation*}

\begin{theorem} \label{thm5}
Bivariate H - K Mittag-Leffler function $E_{\varrho,\Upsilon}^{(\gamma_1;\gamma_2;\gamma_3)}(X,Y)$  has the following operetional representation.
\begin{equation*}
E_{\varrho,\Upsilon}^{(\gamma_1;\gamma_2;\varkappa)}(X,Y)=\frac{X^{1-\varkappa}Y^{1-\varrho}}{[1-D_Y^{-\Upsilon}]^{\gamma_2}}{}_{\phantom{1}3}F_0\bigg[\frac{\gamma_1}{2},\frac{\gamma_1+1}{2},\gamma_2,-,\frac{4}{D_X^{-1}(1-D_Y^{-\Upsilon})}\bigg]\bigg[\frac{X^{\varkappa-1}Y^{\varrho-1}}{\Gamma(\varrho)}\bigg].
\end{equation*}
\end{theorem}

\begin{proof}
The proof is obtained by using the expansion of $(1-D_Y^{-\Upsilon})^{-1}$ and ${}_{\phantom{1}3}F_0(\gamma_1, \gamma_2, \gamma_3; - ; X)$.
\begin{equation*}
\begin{aligned}
&{} \frac{X^{1-\varkappa}Y^{1-\varrho}}{[1-D_Y^{-\Upsilon}]^{\gamma_2}}{}_{\phantom{1}3}F_0\bigg[\frac{\gamma_1}{2},\frac{\gamma_1+1}{2},\gamma_2,-,\frac{4}{D_X^{-1}(1-D_Y^{-\Upsilon})}\bigg]\bigg[\frac{X^{\varkappa-1}Y^{\varrho-1}}{\Gamma(\varrho)}\bigg] \\&
=\frac{X^{1-\varkappa}Y^{1-\varrho}}{[1-D_Y^{-\Upsilon}]^{\gamma_2}}\sum_{s=0}^\infty \frac{(\frac{\gamma_1}{2})_s(\frac{\gamma_1+1}{2})_s(\gamma_2)_s}{s!}\bigg(\frac{4D_X^{-1}}{1-D_Y^{-\Upsilon}}\bigg)^s\bigg[\frac{X^{\varkappa-1}Y^{\varrho-1}}{\Gamma(\varrho)}\bigg] \\&
=X^{1-\varkappa}Y^{1-\varrho}\sum_{s=0}^\infty \frac{2^{2s}(\frac{\gamma_1}{2})_s(\frac{\gamma_1+1}{2})_s(\gamma_2)_s}{s!}\sum_{r=0}^\infty \frac{(\gamma_2+s)_rD_Y^{-\Upsilon r}}{r!}\bigg[\frac{X^{\varkappa-1}Y^{\varrho-1}}{\Gamma(\varrho)}\bigg] \\&
=X^{1-\varkappa}Y^{1-\varrho}\sum_{s=0}^\infty\sum_{r=0}^\infty \frac{(\gamma_1)_{2s}(\gamma_2)_{s+r}}{s!r!}D_X^{-s}D_Y^{-\Upsilon r}\bigg[\frac{X^{\varkappa-1}Y^{\varrho-1}}{\Gamma(\varrho)}\bigg] \\&
=X^{1-\varkappa}Y^{1-\varrho}\sum_{s=0}^\infty\sum_{r=0}^\infty \frac{(\gamma_1)_{2s}(\gamma_2)_{s+r}}{s!r!}\bigg[\frac{X^{s+\varkappa-1}Y^{\Upsilon r+\varrho-1}}{(\varkappa)_s\Gamma(\Upsilon r+\varrho)}\bigg]\\&
=E_{\varrho,\Upsilon}^{(\gamma_1;\gamma_2;\varkappa)}(X,Y).
\end{aligned}
\end{equation*}
\end{proof}

\begin{theorem}
2D - Hermite Konhauser polynomials, $_{\Upsilon}H_n^{\varrho}(X,Y)$, have the following operational representation,
\begin{equation*}
{}_{\phantom{1}\Upsilon}H_n^{\varrho}(X,Y)=[2X(1-D_Y^{-\Upsilon})]^nY^{-\varrho}\sum_{s=0}^{[\frac{n}{2}]}\frac{(\frac{-n}{2})_s(\frac{-n+1}{2})_s}{s!}\bigg[\frac{-1}{X^2(1-D_Y^{-\Upsilon})}\bigg]^s\bigg[\frac{Y^{\varrho}}{\Gamma(\varrho+1)}\bigg].
\end{equation*}
\end{theorem}
\begin{proof}
The proof follows by \eqref{10}.
\end{proof}

As a consequence of the above theorem we can state.

\begin{corollary}
For the 2D - Hermite Konhauser polynomials $_{\Upsilon}H_n^{\varrho}(X,Y)$ the following operational Bessel polynomials representation hold.
\begin{equation*}
{}_{\phantom{1}\Upsilon}H_n^{\varrho}(X,Y)=[2X(1-D_Y^{-\Upsilon})]^nY^{-\varrho}Y_n\bigg(\frac{1}{X^2(1-D_Y^{-\Upsilon})};\frac{3-3n}{2},1\bigg)\bigg[\frac{Y^{\varrho}}{\Gamma(\varrho+1)}\bigg],
\end{equation*}
where $Y_n(X;a,b)={}_{\phantom{1}2}F_0[-n,a-1+n;-,\frac{-X}{b}]$ are the Bessel polynomials.
\end{corollary}

In the proof of the following theorem, we use the following  representations of  the Gamma function. 

\begin{equation}
\Gamma(Y)=\int_0^{\infty}e^{-u}u^{Y-1}du    \qquad  (Re(Y)>0),  \label{han1}
\end{equation}

\begin{equation}
\frac{1}{\Gamma(Y)}=\frac{1}{2\pi i}\int_{-\infty}^{0+}e^{u}u^{-Y}du   \qquad (|arg(u)|\leq \pi ).  \label{han2}
\end{equation}

The second relation is knoıwn as the Hankel's representation \cite{last}. the contour of integration \eqref{han2} is the standard Hankel contour.

\begin{theorem} \label{thm7}
The following integral representation holds for the bivariate  H - K Mittag-Leffler function $E_{\varrho,\Upsilon}^{(\gamma_1;\gamma_2;\gamma_3)}(X,Y)$ .

\begin{equation*}
\begin{aligned}
E_{\varrho,\Upsilon}^{(\gamma_1;\gamma_2;\gamma_3)}(X,Y)=& {}\frac{-\Gamma(\gamma_3)}{4\pi^2\Gamma(\gamma_1)}\int_{-\infty}^{0+}\int_{-\infty}^{0+}\int_0^{\infty}e^{w+t-u}t^{-\varrho}w^{-\gamma_3}u^{\gamma_1-1} \\ & \times \bigg(1-\frac{Y^\Upsilon}{t^\Upsilon}-\frac{Xu^2}{w}\bigg)^{-\gamma_2}dudwdt.
\end{aligned}
\end{equation*}
\end{theorem}

\begin{proof}
Using the  formulas \eqref{han1} and \eqref{han2}, we find that
\begin{equation*}
\begin{aligned}
E_{\varrho,\Upsilon}^{(\gamma_1;\gamma_2;\gamma_3)}(X,Y)= &{} \frac{-\Gamma(\gamma_3)}{4\pi^2\Gamma(\gamma_1)}\int_{-\infty}^{0+}\int_{-\infty}^{0+}\int_0^{\infty}e^{w+t-u}t^{-\varrho}w^{-\gamma_3}u^{\gamma_1-1}\\& \times \sum_{s=0}^\infty\sum_{r=0}^\infty\frac{(\gamma_2)_{s+r}}{s!r!}\bigg(\frac{u^2X}{w}\bigg)^s\bigg(\frac{Y^\Upsilon}{t^\Upsilon}\bigg)^rdudwdt.
\end{aligned}
\end{equation*}
Since $(\gamma)_{s+r}=(\gamma)_s(\gamma+s)_r$, we get
\begin{equation*}
\begin{aligned}
E_{\varrho,\Upsilon}^{(\gamma_1;\gamma_2;\gamma_3)}(X,Y){}& =\frac{-\Gamma(\gamma_3)}{4\pi^2\Gamma(\gamma_1)}\int_{-\infty}^{0+}\int_{-\infty}^{0+}\int_0^{\infty}e^{w+t-u}t^{-\varrho}w^{-\gamma_3}u^{\gamma_1-1}\\& \times\sum_{s=0}^\infty\frac{(\gamma_2)_{s}}{s!}\bigg(\frac{u^2X}{w}\bigg)^s\sum_{r=0}^\infty\frac{(\gamma_2+s)_{r}}{r!}\bigg(\frac{Y^\Upsilon}{t^\Upsilon}\bigg)^rdudwdt
\\ & =\frac{-\Gamma(\gamma_3)}{4\pi^2\Gamma(\gamma_1)}\int_{-\infty}^{0+}\int_{-\infty}^{0+}\int_0^{\infty}e^{w+t-u}t^{-\varrho}w^{-\gamma_3}u^{\gamma_1-1}\bigg(\frac{t^\Upsilon-Y^\Upsilon}{t^\Upsilon}\bigg)^{-\gamma_2} \\& \times \sum_{s=0}^\infty\frac{(\gamma_2)_{s}}{s!}\bigg(\frac{u^2Xt^\Upsilon}{w(t^\Upsilon-Y^\Upsilon)}\bigg)^sdudwdt
\\ & =\frac{-\Gamma(\gamma_3)}{4\pi^2\Gamma(\gamma_1)}\int_{-\infty}^{0+}\int_{-\infty}^{0+}\int_0^{\infty}e^{w+t-u}t^{-\varrho}w^{-\gamma_3}u^{\gamma_1-1}\bigg[1-\frac{u^2Xt^\Upsilon}{w(t^\Upsilon-Y^\Upsilon)}\bigg]^{-\gamma_2}\\& \times \bigg(\frac{t^\Upsilon-Y^\Upsilon}{t^\Upsilon}\bigg)^{-\gamma_2}dudwdt.
\end{aligned}
\end{equation*}
Whence the result.
\end{proof}

\begin{corollary}
The following integral representation holds for the 2D Hermite Konhauser polynomials $_{\Upsilon}H_n^{\varrho}(X,Y)$.
\begin{equation*}
\begin{aligned}
{}_{\phantom{1}\Upsilon}H_n^{\varrho}(X,Y)= &{} \frac{-1}{4\pi^2}\int_{-\infty}^{0+}\int_{-\infty}^{0+}\int_0^{\infty}e^{w-t+u}t^{-1}w^{-1-\varrho}\bigg[\frac{2Xu}{t}\bigg(1-\frac{Y^\Upsilon}{w^\Upsilon}\bigg)\bigg]^n \\& \times \bigg[1+\frac{t^2w^\Upsilon}{4X^2u(w^\Upsilon-Y^\Upsilon)}\bigg]^ndudwdt.
\end{aligned}
\end{equation*}
\end{corollary}
\begin{proof}
The results follows by \eqref{10}.
\end{proof}

\subsection{Laplace transform and Fractional Calculus  for  the  Bivariate H - K Mittag-Leffler Function and the  2D - Hermite Konhauser polynomials}

In this subsection, firstly we obtain the Laplace transform of the $E_{\varrho,\Upsilon}^{(\gamma_1;\gamma_2;\gamma_3)}(X,Y)$ and $_{\Upsilon}H_n^{\varrho}(X,Y)$. Then, we compute the images of these functions under the actions of  Riemann- Liouville fractional integral and derivative operators.

The Laplace transform of the function $f$  is defined by

\begin{equation} 
\mathbb{L}[f](s)=\int_0^{\infty}e^{-st}f(t)dt  \qquad(Re(s)>0). \label{lap}
\end{equation}
 
\begin{theorem} \label{thm8}
For $|\frac{w^\Upsilon}{q^\Upsilon}|<1$, the Laplace transform $E_{\varrho,\Upsilon}^{(\gamma_1;\gamma_2;\gamma_3)}(X,Y)$ is given by
\begin{equation*}
\mathbb{L}[Y^{\varrho-1}E_{\varrho,\Upsilon}^{(\gamma_1;\gamma_2;\frac{\gamma_1+1}{2})}(X,wY)]=\frac{1}{q^\varrho}\bigg[\frac{q^\Upsilon-w^\Upsilon}{q^\Upsilon}\bigg]^{-\gamma_2}{}_{\phantom{1}2}F_0\bigg[\frac{\gamma_1+1}{2},\gamma_2;-;\frac{4q^\Upsilon X}{q^\Upsilon-w^\Upsilon}\bigg].
\end{equation*}
\end{theorem}

\begin{proof}
Using the definition \eqref{lap} and interchanging the order of series and  integral, we obtain
\begin{equation*}
\begin{aligned}
\mathbb{L}[Y^{\varrho-1}E_{\varrho,\Upsilon}^{(\gamma_1;\gamma_2;\frac{\gamma_1+1}{2})}(X,wY)] {} &=\sum_{s=0}^\infty\sum_{r=0}^\infty\frac{(\frac{\gamma_1}{2})_s(\gamma_2)_{s+r}(4X)^sw^{\Upsilon r}}{\Gamma(\varrho+\Upsilon r)s!r!}\int_0^\infty e^{-qY}Y^{\varrho+\Upsilon r+1}dY
\\& =\frac{1}{q^\varrho}\sum_{s=0}^\infty\frac{(\frac{\gamma_1}{2})_s(\gamma_2)_{s}(4X)^s}{s!}\sum_{r=0}^\infty\frac{(\gamma_2+s)_{r}}{r!}\bigg(\frac{w}{q}\bigg)^{\Upsilon r}
\\& =\frac{1}{q^\varrho}\bigg[\frac{q^\Upsilon-w^\Upsilon}{q^\Upsilon}\bigg]^{-\gamma_2}\sum_{s=0}^\infty\frac{(\frac{\gamma_1}{2})_s(\gamma_2)_{s}}{s!}\bigg(\frac{4Xq^\Upsilon}{q^\Upsilon-w^\Upsilon}\bigg)^s.
\end{aligned}
\end{equation*}
Whence the result.
\end{proof}

\begin{corollary}
For $|\frac{w^\Upsilon}{q^\Upsilon}|<1$, the Laplace transfrom of 2D - Hermite Konhauser polynomials is given by
\begin{equation*}
\mathbb{L}[Y^{\varrho}{}_{\phantom{1}\Upsilon}H_n^{\varrho}(X,wY)]=\frac{(2X)^n}{q^{\varrho+1}}\bigg[\frac{q^\Upsilon-w^\Upsilon}{q^\Upsilon}\bigg]^{n}{}_{\phantom{1}2}F_0\bigg[\frac{-n}{2},\frac{-n+1}{2};-;\frac{-q^\Upsilon X^2}{q^\Upsilon-w^\Upsilon}\bigg].
\end{equation*}
\end{corollary}
\begin{proof}
The results follows by \eqref{10}.
\end{proof}

The Riemann - Louville fractional integral of order $\varkappa \in\mathbb{C}$  $(Re(\varkappa)>0, X>A)$ is defined by

\begin{equation*}
_{X}I^{\varkappa}_{A^+}(f)=\frac{1}{\Gamma(\varkappa)}\int_A^X(X-t)^{\varkappa-1}f(t)dt,  \quad f \in L^1[A,B].
\end{equation*}

The Riemann - Louville fractional derivative of order $\varkappa \in\mathbb{C}$  $\big(Re(\varkappa)>0, X>A, n=[Re(\varkappa)]+1\big)$ is defined by

\begin{equation*}
_{X}D^{\varkappa}_{A^+}(f)=\bigg(\frac{d}{dX}\bigg)^n_{X}I^{n-\varkappa}_{A^+}(f),     \quad f  \in C^n[A,B],
\end{equation*}
where $[Re(\varkappa)]$ is the integral part of $Re(\varkappa)$.

\begin{theorem} \label{thm9}
Bivariate H - K  Mittag-Leffler function have the following image under the action of  Riemann - Louville fractional integral operator,
\begin{equation*}
_{Y}I^{\mu}_{B^+}\bigg[(Y-B)^{\varrho-1}E_{\varrho,\Upsilon}^{(\gamma_1;\gamma_2;\gamma_3)}(X,w(Y-B))\bigg]=(Y-B)^{\varrho+\mu-1}E_{\varrho+\mu,\Upsilon}^{(\gamma_1;\gamma_2;\gamma_3)}(X,w(Y-B)).
\end{equation*}
\end{theorem}

\begin{proof}
For $Re(\mu)>0$, 
\begin{equation*}
\begin{aligned}
_{Y}I^{\mu}_{B^+}\bigg[(Y-B)^{\varrho-1}E_{\varrho,\Upsilon}^{(\gamma_1;\gamma_2;\gamma_3)}(X,w(Y-B))\bigg] {}&=\frac{1}{\Gamma(\mu)}\int_b^Y(Y-t)^{\mu-1}(t-B)^{\varrho-1} \\& \times E_{\varrho,\Upsilon}^{(\gamma_1;\gamma_2;\gamma_3)}(X,w(t-B))dt
\\& =\frac{1}{\Gamma(\mu)}\sum_{s=0}^{\infty}\sum_{r=0}^\infty \frac{(\gamma_1)_{2s}(\gamma_2)_{s+r}(X)^s w^{\Upsilon r}}{(\gamma_3)_s\Gamma(\varrho+\Upsilon r)s!r!} \\& \times \int_b^Y (Y-t)^{\mu-1}(t-B)^{\Upsilon r+\varrho-1}dt
\\ & =\sum_{s=0}^{\infty}\sum_{r=0}^\infty \frac{(\gamma_1)_{2s}(\gamma_2)_{s+r}(X)^s (w(Y-B))^{\Upsilon r}}{(\gamma_3)_s\Gamma(\varrho+\Upsilon r+\mu)s!r!} \\&\times (Y-B)^{\mu+\varrho-1}.
\end{aligned}
\end{equation*}
Therefore, for $Re(\varrho)>-1$ and $Re(\mu)>0$, we get the result.
\end{proof}

\begin{corollary}
For 2D - Hermite Konhauser polynomials, we have,
\begin{equation*}
_{Y}I^{\mu}_{B^+}\bigg[(Y-B)^{\varrho}_{\phantom{1}\Upsilon}H_n^{\varrho}(X,w(Y-B))\bigg]=(Y-B)^{\varrho+\mu}
_{\phantom{1}\Upsilon}H_n^{\varrho+1}(X,w(Y-B)).
\end{equation*}
 \end{corollary}
\begin{proof}
The results follows by \eqref{10}.
\end{proof}

\begin{theorem} \label{thm10}
Bivariate H - K Mittag-Leffler function have the following  image under the action Riemann - Louville fractional derivative operator,
\begin{equation*}
_{Y}D^{\mu}_{B^+}\bigg[(Y-B)^{\varrho-1}E_{\varrho,\Upsilon}^{(\gamma_1;\gamma_2;\gamma_3)}(X,w(Y-B))\bigg]=(Y-B)^{\varrho-\mu-1}E_{\varrho-\mu,\Upsilon}^{(\gamma_1;\gamma_2;\gamma_3)}(X,w(Y-B)).
\end{equation*}
\end{theorem}

\begin{proof}
For $Re(\mu)\geq 0$ , 
\begin{equation*}
\begin{aligned}
{}&  _{Y}D^{\mu}_{B^+}\bigg[(Y-B)^{\varrho-1}E_{\varrho,\Upsilon}^{(\gamma_1;\gamma_2;\gamma_3)}(X,w(Y-B))\bigg] \\&= D_Y^n {}_{\phantom{1}Y}I^{n-\mu}_{B^+}\bigg[(Y-B)^{\varrho-1}E_{\varrho,\Upsilon}^{(\gamma_1;\gamma_2;\gamma_3)}(X,w(Y-B))\bigg]
\\& =\frac{1}{\Gamma(n-\mu)}\sum_{s=0}^{\infty}\sum_{r=0}^\infty \frac{(\gamma_1)_{2s}(\gamma_2)_{s+r}(X)^s w^{\Upsilon r}}{(\gamma_3)_s\Gamma(\varrho+\Upsilon r)s!r!}
\\ & \times D_Y^n\int_b^Y (Y-t)^{n-\mu-1}(t-B)^{\Upsilon r+\varrho-1}dt
\\&=\sum_{s=0}^{\infty}\sum_{r=0}^\infty \frac{(\gamma_1)_{2s}(\gamma_2)_{s+r}(X)^s (w(Y-B))^{\Upsilon r}}{(\gamma_3)_s\Gamma(\varrho-\mu-\Upsilon r)s!r!} \\& \times (Y-B)^{\varrho-\mu-1}
\\& =(Y-B)^{\varrho-\mu-1}E_{\varrho-\mu,\Upsilon}^{(\gamma_1;\gamma_2;\gamma_3)}(X,w(Y-B)).
\end{aligned}
\end{equation*}
\end{proof}

\begin{corollary}
For 2D - Hermite Konhauser polynomials, we have,
\begin{equation*}
_{Y}D^{\mu}_{B^+}\bigg[(Y-B)^{\varrho}{}_{\phantom{1}\Upsilon}H_n^{\varrho}(X,w(Y-B))\bigg]=(Y-B)^{\varrho-\mu}
{}_{\phantom{1}\Upsilon}H_n^{\varrho-\mu}(X,w(Y-B)).
\end{equation*}
 \end{corollary}
\begin{proof}
The results follows by \eqref{10}.
\end{proof}

\section{ The Modified 2D - Hermite Konhauser polynomials}

In this section in order to construct fractional calculus operators having the semigroup property (and therefore have the left inverse operator), we  modify the 2D - Hermite Konhsauser polynomials by adding two new parameters. Furthermore corresponding to this modified polynomials we introduced the modified bivariate H - K Mittag-Leffler function.

\begin{definition}
The modified 2D - Hermite Konhauser polynomials are defined by the following representation
\begin{equation}
{}_{\phantom{1}\Upsilon}H_n^{\varkappa,\varrho; c}(X,Y)=\sum_{s=0}^{[\frac{n}{2}]}\sum_{r=0}^{n-s}\frac{(-1)^s(-n)_{2s}(-n)_{s+r}}{(-n)_s(c)_s\Gamma(\varkappa+1+s)\Gamma(\varrho+1+\Upsilon r)s!r!}(2X)^{n-2s}Y^{\Upsilon r} \label{h2}
\end{equation}
where  $, \varkappa>-1 , \varrho>-1$ and $\Upsilon=1,2\dots$.
\end{definition}

\begin{corollary}
The following Kampe de Feriet's double hypergeometric representation  holds true for the modified 2D Hermite Konhauser polynomials ${}_{\phantom{1}\Upsilon}H_n^{\varkappa,\varrho;{\varkappa+1}}(X,Y)$,
\begin{equation*}
{}_{\phantom{1}\Upsilon}H_n^{\varkappa,\varrho; c}(X,Y)=\frac{(2X)^n}{\Gamma(1+\varkappa)\Gamma(1+\varrho)}F_{0,\Upsilon,2}^{1,0,3}\begin{bmatrix} -n:-;\bigtriangleup(2;-n);&\\ & \frac{-1}{X^2};\bigg(\frac{Y}{\Upsilon}\bigg)^\Upsilon \\ - :\bigtriangleup(\Upsilon;\varrho+1);-n,c,\varkappa+1; & \end{bmatrix}.
\end{equation*}
\end{corollary}
\begin{proof}
The proof follows in a similar way as it is in the proof of Theorem \ref{thm4}.
\end{proof}
Corresponding to this modified 2D - Hermite Konhauser polynomials, we introduce the modified bivariate H - K Mittag-Leffler function $E_{\varkappa,\varrho,\Upsilon}^{(\gamma_1;\gamma_2; \gamma_3,\gamma_4)}(X,Y)$ below.

\begin{definition}
The modified bivariate H - K Mittag-Leffler function $E_{\varrho,\Upsilon}^{(\gamma_1;\gamma_2;\gamma_3;\gamma_4)}(X,Y)$  is defined by the formula,
\begin{equation}
E_{\varkappa,\varrho,\Upsilon}^{(\gamma_1;\gamma_2;\gamma_3;\gamma_4)}(X,Y)=\sum_{s=0}^\infty\sum_{r=0}^\infty\frac{(\gamma_1)_{2s}(\gamma_2)_{s+r}X^sY^{\Upsilon r}}{(\gamma_3)_s(\gamma_4)_s\Gamma(\varkappa+s)\Gamma(\varrho+\Upsilon r)r!s!},   \label{e2}
\end{equation}
$(\Upsilon, \varkappa, \varrho,\gamma_1,\gamma_2,\gamma_3 ,\gamma_4 \in\mathbb{C}, Re(\gamma_1)>0, Re(\gamma_2)>0, Re(\gamma_3)>0, Re(\gamma_4)>0, Re(\varkappa)>0,  Re(\varrho), Re(\Upsilon)>0).$
\end{definition}

By choosing $ X=0$ in \eqref{e2} we get,
\begin{equation*}
E_{\varkappa,\varrho,\Upsilon}^{(\gamma_1;\gamma_2;\gamma_3;\gamma_4)}(0,Y)=\frac{1}{\Gamma(\varkappa)}E_{\Upsilon,\varrho}^{\gamma_2}(Y^\Upsilon).
\end{equation*}

\begin{corollary}
Comparing \eqref{h2} and \eqref{e2} it  is easily seen that;
\begin{equation}
{}_{\phantom{1}\Upsilon}H_n^{\varkappa,\varrho;c}(X,Y)=(2X)^nE_{\varkappa+1,\varrho+1,\Upsilon}^{(-n;-n;-n;c)}\bigg(\frac{-1}{4X^2},Y\bigg). \label{18}
\end{equation}
\end{corollary}

\begin{corollary}
Comparing \eqref{e2} and \eqref{cemo} it is  easily seen that;
\begin{equation*}
E_{\varkappa,\varrho,\Upsilon}^{(\gamma_1;\gamma_2;\frac{\gamma_1}{2};\frac{\gamma_1+1}{2})}(\frac{X}{4},Y)=E_{\varkappa,\varrho,\Upsilon}^{\gamma_2}(X,Y).
\end{equation*}
\end{corollary}

\subsection{Operational and Integral  Representation for the   Modified Bivariate H - K Mittag-Leffler Function and modified 2D - Hermite Konhauser polynomials}

In this subsection,   we first present the operational  representations for the modified bivariate H - K Mittag-Leffler function $E_{\varrho,\Upsilon}^{(\gamma_1;\gamma_2;\gamma_3;\gamma_4)}(X,Y)$ and  the modified 2D - Hermite Konhauser polynomials $E_{\varrho,\Upsilon}^{(\gamma_1;\gamma_2;\gamma_3)}(X,Y)$. Moreover the integral representations of $E_{\varkappa,\varrho,\Upsilon}^{(\gamma_1;\gamma_2;\gamma_3;\gamma_4)}(X,Y)$ and ${}_{\phantom{1}\Upsilon}H_n^{\varkappa,\varrho;c}(X,Y)$  are also derived.

\begin{theorem}
Modified bivariate H - K Mittag-Leffler function $E_{\varrho,\Upsilon}^{(\gamma_1;\gamma_2;\gamma_3;\gamma_4)}(X,Y)$  have the following operetional representation.
\begin{equation*}
\begin{aligned}
E_{\varrho,\Upsilon}^{(\gamma_1;\gamma_2;\gamma_3;\gamma_4)}(X,Y)= &{} \frac{X^{1-\varkappa}Y^{1-\varrho}}{[1-D_Y^{-\Upsilon}]^{\gamma_2}}{}_{\phantom{1}3}F_2\bigg[\frac{\gamma_1}{2},\frac{\gamma_1+1}{2},\gamma_2;\gamma_3,\gamma_4;\frac{4}{D_X^{-1}(1-D_Y^{-\Upsilon})}\bigg]
\\ & \times \bigg[\frac{X^{\varkappa-1}Y^{\varrho-1}}{\Gamma(\varrho)\Gamma(\varkappa)}\bigg],
\end{aligned}
\end{equation*}
where 
\begin{equation*}
{}_{\phantom{1}3}F_2(\gamma_1, \gamma_2, \gamma_3; \varrho_1, \varrho_2; X)=\sum_{n=0}^\infty\frac{(\gamma_1)_n(\gamma_2)_n(\gamma_3)_n}{(\varrho_1)_n(\varrho_2)_n}\frac{X^n}{n!}.
\end{equation*}
\end{theorem}
\begin{proof}
The proof follows in line of the proof of  Theorem \ref{thm5}.
\end{proof}

\begin{corollary}
For $\gamma_3=\frac{\gamma_1}{2}$ and $\gamma_4=\frac{\gamma_1+1}{2}$  the following operational representation also holds
\begin{equation*}
E_{\varrho,\Upsilon}^{(\gamma_1;\gamma_2;\frac{\gamma_1}{2};\frac{\gamma_1+1}{2})}(X,Y)= X^{1-\varkappa}Y^{1-\varrho}\bigg[\frac{D_X^{-1}}{D_X^{-1}(1-D_Y^{-\Upsilon})-4}\bigg]^{\gamma_2}\bigg[\frac{X^{\varkappa-1}Y^{\varrho-1}}{\Gamma(\varrho)\Gamma(\varkappa)}\bigg].
\end{equation*}
\end{corollary}

\begin{corollary}
Modified 2D - Hermite Konhauser polynomials, ${}_{\phantom{1}\Upsilon}H_n^{\varkappa,\varrho; c}(X,Y)$, have the following operational representation.
\begin{equation*}
\begin{aligned}
{}& {}_{\phantom{1}\Upsilon}H_n^{\varkappa,\varrho; c}(X,Y) \\&
= \frac{[2X(1-D_Y^{-\Upsilon})]^nY^{-\varrho}}{\Gamma(1+\varkappa)}{}_{\phantom{1}2}F_2\bigg[\frac{-n}{2},\frac{-n+1}{2};c,\varkappa+1;\frac{-1}{X^2(1-D_Y^{-\Upsilon})}\bigg]\bigg[\frac{Y^{\varrho}}{\Gamma(\varrho+1)}\bigg],
 \end{aligned}
\end{equation*}
%which can be written in terms of Bessel polynomials,
%\begin{equation*}
%{}_{\phantom{1}\Upsilon}H_n^{\varkappa,\varrho; \varkappa+1}(X,Y)= \frac{[2X(1-D_Y^{-\Upsilon}]^nY^{-\varrho}}{\Gamma(1+\varkappa)}Y_n\bigg(\frac{1}{X^2(1-D_Y^{-\Upsilon})};\frac{3-3n}{2},1\bigg)\bigg[\frac{Y^{\varrho}}{\Gamma(\varrho+1)}\bigg].
%\end{equation*}
\end{corollary}

\begin{theorem}
The following integral representation holds for the modified bivariate  H -K Mittag-Leffler function $E_{\varkappa,\varrho,\Upsilon}^{(\gamma_1;\gamma_2;\gamma_3;\gamma_4)}(X,Y)$ .

\begin{equation*}
\begin{aligned}
E_{\varkappa,\varrho,\Upsilon}^{(\gamma_1;\gamma_2;\gamma_3;\gamma_4)}(X,Y)={}&\frac{\Gamma(\gamma_3)\Gamma(\gamma_4)}{16\pi^4\Gamma(\gamma_1)}\int_{-\infty}^{0+}\int_{-\infty}^{0+}\int_{-\infty}^{0+}\int_{-\infty}^{0+}\int_0^{\infty}e^{-u+t+w+v+z}t^{-\gamma_4}w^{-\gamma_3}u^{\gamma_1-1} \\& \times v^{-\varkappa}z^{-\varrho}\bigg(\frac{z^{\Upsilon}-Y^\Upsilon}{z^\Upsilon}-\frac{Xu^2}{uvz}\bigg)^{-\gamma_2}dzdwdvdtdu.
\end{aligned}
\end{equation*}
\end{theorem}
\begin{proof}
The proof follows in a similar way as it is in the proof of  Theorem \ref{thm7}.
\end{proof}

\begin{corollary}
For the modified 2D Hermite Konhauser polynomials, we have
\begin{equation*}
\begin{aligned}
{}_{\phantom{1}\Upsilon}H_n^{\varkappa,\varrho;c}(X,Y)={}& \frac{\Gamma(c)}{16\pi^4}\int_0^\infty\int_{-\infty}^{0+}\int_{-\infty}^{0+}\int_{-\infty}^{0+}\int_{-\infty}^{0+}e^{-t+w+u+v+z}t^{-1}w^{-1-\varrho}v^{-c}z^{-1-\varkappa}\\&
\times \bigg[\frac{2Xu}{t}\bigg(\frac{w^\Upsilon-Y^\Upsilon}{w^\Upsilon}\bigg)\bigg(1+\frac{t^2}{4X^2uvz}\bigg)\bigg]^ndwdudvdzdt.
\end{aligned}
\end{equation*}
\end{corollary}

\subsection{Laplace transform and Fractional Calculus  for  the Modified Bivariate H - K Mittag-Leffler Function and the Modified 2D - Hermite Konhauser polynomials}
In this subsection,we first obtain the Laplace transforms of the $E_{\varrho,\Upsilon}^{(\gamma_1;\gamma_2;\gamma_3;\gamma_4)}(X,Y)$ and ${}_{\phantom{1}\Upsilon}H_n^{\varkappa,\varrho;c}(X,Y)$. We then compute the images of these functions under the action of the double Riemann–Liouville fractional integral and derivative operators.

\begin{theorem}
For $|\frac{w^\Upsilon}{q^\Upsilon}|<1$, the following Laplace transform holds for $E_{\varrho,\Upsilon}^{(\gamma_1;\gamma_2;\gamma_3;\gamma_4)}(X,Y)$ when $\gamma_3=\frac{\gamma_1+1}{2}$ and $\gamma_4=\gamma_2$.
\begin{equation*}
\mathbb{L}[Y^{\varrho-1}E_{\varkappa,\varrho,\Upsilon}^{(\gamma_1;\gamma_2;\frac{\gamma_1+1}{2};\gamma_2)}(X,wY)]=\frac{1}{q^\varrho}\bigg[\frac{q^\Upsilon-w^\Upsilon}{q^\Upsilon}\bigg]^{-\gamma_2}E_{1,\varkappa+1}^{\frac{\gamma_1}{2}}\bigg(\frac{4q^{\Upsilon}X}{q^{\Upsilon}-w^{\Upsilon}}\bigg).
\end{equation*}
\end{theorem}
\begin{proof}
The proof follows in a similar way as it is in the proof of  Theorem \ref{thm8}.
\end{proof}

Now, recall the bivariate Laplace transform

\begin{equation*}
\mathbb{L}_2[h(\bar{x},\bar{y})](\bar{p},\bar{q})=\int_0^\infty\int_0^\infty e^{-(\bar{p}\bar{x}+\bar{q}\bar{y})}h(\bar{x},\bar{y})d\bar{x}d\bar{y},   \quad  (Re(\bar{p}), Re(\bar{q})>0).
\end{equation*}

\begin{theorem} \label{thm14}
For $|\frac{\nu_2^\Upsilon}{q^\Upsilon}|<1$ and $|\frac{\nu_1^\Upsilon q^\Upsilon}{p(q^\Upsilon-\nu_2^\Upsilon)}|<1$, the following Laplace transform holds for $E_{\varrho,\Upsilon}^{(\gamma_1;\gamma_2;\gamma_3)}(X,Y)$ when $\gamma_3=\frac{\gamma_1+1}{2}$ and $\gamma_4=\gamma_2$.
\begin{equation*}
\mathbb{L}_2[X^{\varkappa-1}Y^{\varrho-1}E_{\varkappa,\varrho,\Upsilon}^{(\gamma_1;\gamma_2;\frac{\gamma_1+1}{2};\gamma_2)}(\nu_1X,\nu_2Y)]=\frac{1}{q^\varrho p^\varkappa}\bigg[\frac{q^\Upsilon-\nu_2^\Upsilon}{q^\Upsilon}\bigg]^{-\gamma_2}\bigg[1-\frac{4\nu_1q^\Upsilon}{p(q^\Upsilon-\nu_2^\Upsilon)}\bigg ]^{-\frac{\gamma_1}{2}}.
\end{equation*}
\end{theorem}
\begin{proof}
The proof follows in a similar way as it is in the proof of  Theorem \ref{thm8}.
\end{proof}

\begin{corollary}
For $|\frac{\nu_2^\Upsilon}{q^\Upsilon}|<1$ and $|\frac{\nu_1^\Upsilon q^\Upsilon}{p(q^\Upsilon-\nu_2^\Upsilon)}|<1$, the following Laplace transform holds for ${}_{\phantom{1}\Upsilon}H_n^{\varkappa,\varrho;c}(X,Y)$ when $c=\frac{-n+1}{2}$.
\begin{equation*}
\mathbb{L}_2[X^{\varkappa+\frac{n}{2}}Y^{\varrho}{}_{\phantom{1}\Upsilon}H_n^{\varkappa,\varrho;\frac{-n+1}{2}}(\frac{\nu_1}{2}X^{\frac{-1}{2}},\nu_2Y)]=\frac{1}{q^{\varrho+1} p^{\varkappa+1}}\bigg[\frac{q^\Upsilon-\nu_2^\Upsilon}{q^\Upsilon}\bigg]^{n}\bigg[1+\frac{4\nu_1q^\Upsilon}{p(q^\Upsilon-\nu_2^\Upsilon)}\bigg ]^{\frac{n}{2}}. \label{lap2}
\end{equation*}
\end{corollary}

The double fractional  integral  of the function $g(X,Y)$ is defined by
\begin{equation*}
({}_{\phantom{1}Y}I_{B^+}^{\zeta}{}_{\phantom{1}Y}I_{A^+}^{\mu})g(X,Y)=\frac{1}{\Gamma(\zeta)\Gamma(\mu)}\int_b^Y\int_a^X (X-\epsilon)^{\mu-1}(Y-\Omega)^{\zeta-1}g(\epsilon,\Omega)d\epsilon  d\Omega,
\end{equation*}
$(X>A ,Y>B, Re(\zeta)>0, Re(\mu)>0).$

The partial Riemann-Liouville  fractional derivative is defined by,
\begin{equation*}
({}_{\phantom{1}X}D_{A^+}^{\zeta}{}_{\phantom{1}Y}D_{B^+}^{\Omega})g(X,Y)=\frac{\partial^n}{\partial X^n}\frac{\partial^m}{\partial Y^m}({}_{\phantom{1}X}I_{A^+}^{n-\zeta}{}_{\phantom{1}Y}I_{B^+}^{m-\Omega})g(X,Y),
\end{equation*}
 $(X>A, Y>B, m=[Re(\Omega)]+1, n=[Re(\zeta)]+1).$

\begin{theorem}
Modified bivariate H - K Mittag-Leffler function have the following double fractional integral representation;
\begin{equation*}
\begin{aligned}
\bigg({}_{\phantom{1}Y}I_{B^+}^{\zeta} {}& {}_{\phantom{1}X}I_{A^+}^{\mu}\bigg)\bigg[(X-A)^{\varkappa-1} (Y-B)^{\varrho-} E_{\varkappa,\varrho,\Upsilon}^{(\gamma_1;\gamma_2;\gamma_3;\gamma_4)}( \varpi_1(X-A), \varpi_2(Y-B))\bigg] \\ &
=(X-A)^{\varkappa+\mu-1} (Y-B)^{\varrho+\zeta-1}E_{\varkappa+\mu,\varrho+\zeta,\Upsilon}^{(\gamma_1;\gamma_2;\gamma_3;\gamma_4)}( \varpi_1(X-A), \varpi_2(Y-B)).
\end{aligned}
\end{equation*}
\end{theorem}
\begin{proof}
The proof follows similarly as in the proof of  Theorem \ref{thm9}.
\end{proof}

\begin{corollary}
The following fractional integral representation holds true for the modified 2D Hermite Konhauser polynomials.
\begin{equation*}
{}_{\phantom{1}Y}I^{\mu}_{B^+}\bigg[(Y-B)^{\varrho}{}_{\phantom{1}\Upsilon}H_n^{\varkappa,\varrho; c}(X,w(Y-B))\bigg]=(Y-B)^{\varrho+\mu}{}_{\phantom{1}\Upsilon}H_n^{\varkappa,\varrho+\mu;c}(X,w(Y-B)).
\end{equation*}
\end{corollary}

\begin{theorem}
Modified bivariate H - K  Mittag-Leffler function have the following partial fractional derivative representation;
\begin{equation*}
\begin{aligned}
\bigg({}_{\phantom{1}Y}D_{B^+}^{\zeta} {}& {}_{\phantom{1}X}D_{A^+}^{\mu}\bigg)\bigg[(X-A)^{\varkappa-1} (Y-B)^{\varrho-1} E_{\varkappa,\varrho,\Upsilon}^{(\gamma_1;\gamma_2;\gamma_3;\gamma_4)}( \varpi_1(X-A), \varpi_2(Y-B))\bigg] \\ &
=(X-A)^{\varkappa-\mu-1} (Y-B)^{\varrho-\zeta-1}E_{\varkappa-\mu,\varrho-\zeta,\Upsilon}^{(\gamma_1;\gamma_2;\gamma_3;\gamma_4)}( \varpi_1(X-A), \varpi_2(Y-B)).
\end{aligned}
\end{equation*}
\end{theorem}
\begin{proof}
The proof follows  in a similar way as it is in the proof of  Theorem \ref{thm10}.
\end{proof}

\begin{corollary}
The following fractional derivative representation holds true for the modified 2D Hermite Konhauser polynomials.
\begin{equation*}
{}_{\phantom{1}Y}D^{\varkappa}_{B^+}\bigg[(Y-B)^{\varrho}{}_{\phantom{1}\Upsilon}H_n^{\varkappa,\varrho;c}(X,w(Y-B))\bigg]=(Y-B)^{\varrho-\mu}{}_{\phantom{1}\Upsilon}H_n^{\varkappa,\varrho-\mu;c}(X,w(Y-B)).
\end{equation*}
\end{corollary}

\subsection{An integral equation containing the modified 2D - Hermite Konhauser polynomials and an integral operator involving Modified bivariate H - K  Mittag-Leffler in the kernel}

 In this subsection, we first  introduce a convolution type integral equation invoving ${}_{\phantom{1}\Upsilon}H_n^{\varkappa,\varrho; c}(X,Y)$. 
Recall that the double fractional integral $({}_{\phantom{1}X}I_{0^+}^{\zeta}{}_{\phantom{1}Y}I_{0^+}^{\mu}g)(X,Y)$ can be written as a convolution of the form
\begin{equation*}
({}_{\phantom{1}X}I_{0^+}^{\zeta}{}_{\phantom{1}Y}I_{0^+}^{\mu})g(X,Y)=\bigg[g(X,Y)* \frac{X_t^{\zeta-1}Y_\Omega^{\mu-1}}{\Gamma(\zeta)\Gamma(\mu)}\bigg]    \quad   (Re(\mu)>0, Re(\zeta)>0).
\end{equation*}
Therefore,
\begin{equation*}
\mathbb{L}_2({}_{\phantom{1}X}I_{0^+}^{\zeta}{}_{\phantom{1}Y}I_{0^+}^{\mu}g)(p,q)=p^{-\zeta}q^{-\mu}L_2(g)(p,q). 
\end{equation*}

In this subsection we consider the integral equation 
\begin{equation}
\int_0^X \int_0^Y(X-t)^{\varkappa+\frac{n}{2}}(Y-\Omega)^\varrho{}_{\phantom{1}\Upsilon}H_n^{\varkappa,\varrho;\frac{-n+1}{2}}\bigg(\frac{\nu_1}{2}(X-t)^{\frac{-1}{2}},\nu_2(Y-\Omega)\bigg)\Phi(t,\Omega)d\Omega dt=\theta(X,Y).   \label{int}
\end{equation}

\begin{theorem} \label{thm17}
The double integral equation \eqref{int} admits the solution is given by
\begin{equation*}
\begin{aligned}
\Phi(X,Y)= &{} \int_0^X \int_0^Y(X-t)^{\zeta-\varkappa-2}(Y-\Omega)^{\mu-\varrho-2}E_{\zeta-\varkappa-1,\mu-\varrho-1,\Upsilon}^{(-n;-n;\frac{-n+1}{2};-n)}\big(\nu_1(X-t),\nu_2(Y-\Omega)\big) \\&
\times{}_{\phantom{1}X}I_{0^+}^{\zeta}{}_{\phantom{1}Y}I_{0^+}^{\mu}\theta(t,\Omega)d\Omega dt.
\end{aligned}
\end{equation*}
\end{theorem}

\begin{proof}
Applying $\mathbb{L}_2$ on  \eqref{int}, using convolution theorem and Theorem \ref{thm14}  we get
\begin{equation*}
\frac{1}{q^{\varrho+1}p^{\varkappa+1}}\bigg[\frac{q^\Upsilon-\nu_2^\Upsilon}{q^\Upsilon}\bigg]^n\bigg[1-\frac{4\nu_1q^\Upsilon}{p(q^\Upsilon-\nu_2^\Upsilon}\bigg]^{\frac{n}{2}}L_2[\Phi(t,\Omega)](p,q)=\mathbb{L}_2[\theta(t,\Omega)](p,q)
\end{equation*}
which gives
\begin{equation*}
\begin{aligned}
\mathbb{L}_2[\Phi(t,\Omega)](p,q)= &{} \frac{1}{q^{\mu-\varrho-1}p^{\zeta-\varkappa-1}}
\frac{1}{q^{\varrho+1}p^{\varkappa+1}}\bigg[\frac{q^\Upsilon-\nu_2^\Upsilon}{q^\Upsilon}\bigg]^{-n}\bigg[1-\frac{4\nu_1q^\Upsilon}{p(q^\Upsilon-\nu_2^\Upsilon}\bigg]^{-\frac{n}{2}} \\& \times p^\zeta q^\mu\mathbb{L}_2[\theta(t,\Omega)](p,q).
\end{aligned}
\end{equation*}
Taking inverse double Laplace transform of the above equation and considering
\begin{equation*}
p^{\zeta}q^{\mu}[g(t,\Omega)](p,q)=\mathbb{L}_2\big[{}_{\phantom{1}X}I_{0^+}^{-\zeta}{}_{\phantom{1}Y}I_{0^+}^{-\mu}[g(t,\Omega)]\big](p,q),
\end{equation*}
we get the desired result.
\end{proof}

Now, we consider the following integral operator special case of which appears in Theorem \ref{thm17}, 
\begin{equation}
\begin{aligned}
\bigg({}_{\phantom{1}{A^+,C^+}}\mathbb{I}_{\varkappa,\varrho, \Upsilon,\nu_1,\nu_2}^{\gamma_1;\gamma_2;\gamma_3;\gamma_4}\psi\bigg)(X,Y)= &{} \int_C^Y \int_A^X(X-t)^{\varkappa-1}(Y-\Omega)^{\varrho-1} \\& \times E_{\varkappa,\varrho, \Upsilon}^{\gamma_1;\gamma_2;\gamma_3;\gamma_4}(\nu_1(X-t),\nu_2(Y-\Omega))\psi(t,\Omega)dtd\Omega.   \label{into}
\end{aligned}
\end{equation}
$(X>A, Y>C)$. \\
In the case $\nu_1=\nu_2=0$, the integral operator ${}_{\phantom{1}{A^+,C^+}}\mathbb{I}_{\varkappa,\varrho, \Upsilon,0,0}^{\gamma_1;\gamma_2;\gamma_3;\gamma_4}$ reduces to the Riemann - Liouville double fractional integral operator.
In the rest of this section we investigate the main properties including transformation, semigroup, composition and left inverse operator.\\
The space $L_1((A,B) \times (C,D))$ of absolutely  integrable functions is defined by

\begin{equation*}
L_1((A,B)\times  (C,D))=\{f: ||f||_1:=\int_A^B\int_C^D|f(X,Y)|dYdX<\infty\}.
\end{equation*}

In the below theorem, we show that the operator ${}_{\phantom{1}{A^+,C^+}}\mathbb{I}_{\varkappa,\varrho, \Upsilon,\nu_1,\nu_2}^{\gamma_1;\gamma_2;\gamma_3;\gamma_4}$ is a transformation from $L_1((A,B) \times (C,D))$ to $L_1((A,B) \times (C,D))$.
\begin{theorem}
The double integral operator ${}_{\phantom{1}{A^+,C^+}}\mathbb{I}_{\varkappa,\varrho, \Upsilon,\nu_1,\nu_2}^{\gamma_1;\gamma_2;\gamma_3;\gamma_4}$ is bounded on the space $L_1((A,B) \times (C,D))$,
\begin{equation}
\bigg|\bigg|{}_{\phantom{1}{A^+,C^+}}\mathbb{I}_{\varkappa,\varrho, \Upsilon,\nu_1,\nu_2}^{\gamma_1;\gamma_2;\gamma_3;\gamma_4}\psi\bigg|\bigg|_1\leq K||\psi||_1,
\end{equation}
where the constant K is independent of $\psi$.
\end{theorem}
\begin{proof}
By using the definition of the operator \eqref{into} and the norm on  the space $L_1((A,B) \times (C,D))$  for the function $\psi \in L_1((A,B) \times (C,D))$, and applying the Fubini's theorem, we find
\begin{equation*}
\begin{aligned}
\bigg|\bigg|{}_{\phantom{1}{A^+,C^+}}\mathbb{I}_{\varkappa,\varrho, \Upsilon,\nu_1,\nu_2}^{\gamma_1;\gamma_2;\gamma_3;\gamma_4}\psi\bigg|\bigg|_1
\leq{}&  \int_A^B\int_C^D|\psi(t,\Omega)|\\ & \times \bigg(\int_t^B \int_\Omega^D(X-t)^{Re(\varkappa)-1}(Y-\Omega)^{Re(\varrho)-1}
\\& \times \big|E_{\varkappa,\varrho, \Upsilon}^{\gamma_1;\gamma_2;\gamma_3;\gamma_4}(\nu_1(X-t),\nu_2(Y-\Omega)\big|dYdX\bigg)d\Omega dt
\\ & =   \int_0^{B-t} \int_0^{D-\Omega}u^{Re(\varkappa)-1}v^{Re(\varrho)-1}|E_{\varkappa,\varrho, \Upsilon}^{\gamma_1;\gamma_2;\gamma_3;\gamma_4}(\nu_1u,\nu_2v\big|dudv\bigg) \\& \times \int_A^B\int_C^D|\psi(t,\Omega)|d\Omega dt
\\& \leq \sum_{s=0}^\infty \sum_{r=0}^\infty\frac{|(\gamma_1)_{2s}||(\gamma_2)_{s+r}||\nu_1|^s|\nu_2|^{\Upsilon r}}{|(\gamma_3)_{s}||(\gamma_4)_{s}||\Gamma(\varkappa+r)|\Gamma(\varrho+\Upsilon r)|r!s!}
\\ & \times \int_0^{B-A} u^{Re(\varkappa)+s-1}du\int_0^{D-C}v^{Re(\varrho+\Upsilon r-1}dv||\psi||_1
\\ & =K||\psi||_1,
\end{aligned}
\end{equation*} 
where 
\begin{equation*}
\begin{aligned}
K= {} & \sum_{s=0}^\infty \sum_{r=0}^\infty\frac{|(\gamma_1)_{2s}||(\gamma_2)_{s+r}||\nu_1(B-A)|^s|\nu_2(D-C)|^{\Upsilon r}}{\big[(Re(\varkappa)+r)(Re(\varrho)+\Upsilon r)\big]|(\gamma_3)_{s}||(\gamma_4)_{s}||\Gamma(\varkappa+r)|\Gamma(\varrho+\Upsilon r)|r!s!} \\& \times (B-A)^{Re(\varkappa)}(D-C)^{Re(\varrho)}.
\end{aligned}
\end{equation*}
\end{proof}

\begin{theorem} \label{19}
The following double  integral transformation holds true
\begin{equation*}
\begin{aligned}
{}&\int_0^X \int_0^Y (X-t)^{\varkappa-1} (Y-\Omega)^{\varrho-1} E_{\varkappa,\varrho,\Upsilon}^{(\gamma_1;\gamma_2;\frac{\gamma_1+1}{2};\gamma_2)}( \varpi_1(X-t), \varpi_2(Y-\Omega))t^{\sigma-1}\Omega^{\eta-1}
\\& \times E_{\sigma,\eta,\Upsilon}^{(\mu_1;\mu_2;\frac{\mu_1+1}{2};\mu_2)}( \varpi_1t, \varpi_2\Omega)d\Omega dt =X^{\varkappa+\sigma-1}Y^{\varrho+\eta-1}E_{\varkappa+\sigma,\varrho+\eta,\Upsilon}^{(\gamma_1+\mu_1;\gamma_2+\mu_2;\frac{\gamma_1+\mu_1+1}{2};\gamma_2+\mu_2)}( \varpi_1X, \varpi_2Y).
\end{aligned}
\end{equation*}
\end{theorem}

\begin{proof}
By using double convolution theorem for the double Laplace transform we get,
\begin{equation*}
\begin{aligned}
{}&\mathbb{L}_2\bigg[\int_0^X  \int_0^Y (X-t)^{\varkappa-1} (Y-\Omega)^{\varrho-1} E_{\varkappa,\varrho,\Upsilon}^{(\gamma_1;\gamma_2;\frac{\gamma_1+1}{2};\gamma_2)}( \varpi_1(X-t), \varpi_2(Y-\Omega))t^{\sigma-1}\Omega^{\eta-1} \\& \times E_{\sigma,\eta,\Upsilon}^{(\mu_1;\mu_2;\frac{\mu_1+1}{2};\mu_2)}( \varpi_1t, \varpi_2\Omega)d\Omega dt\bigg](p,q)
\\& =\mathbb{L}_2\big[X^{\varkappa-1}Y^{\varrho-1}E_{\varkappa,\varrho,\Upsilon}^{(\gamma_1;\gamma_2;\frac{\gamma_1+1}{2};\gamma_2)}( \varpi_1 X, \varpi_2 Y)\big](p,q ) \\ & \times \mathbb{L}_2\big[X^{\sigma-1}Y^{\eta-1}E_{\sigma,\eta,\Upsilon}^{(\mu_1;\mu_2;\frac{\mu_1+1}{2};\mu_2)}( \varpi_1 X, \varpi_2 Y)\big](p,q).
\end{aligned}
\end{equation*}
From Theorem \ref{19}, we have for $Re(\varkappa), Re(\varrho)>0$ and $|\frac{\nu_2^\Upsilon}{q^\Upsilon}|<1$ and $|\frac{\nu_1^\Upsilon q^\Upsilon}{p(q^\Upsilon-\nu_2^\Upsilon)}|<1$, that
\begin{equation*}
\begin{aligned}
\mathbb{L}_2\bigg[{}& \int_0^X  \int_0^Y (X-t)^{\varkappa-1} (Y-\Omega)^{\varrho-1} E_{\varkappa,\varrho,\Upsilon}^{(\gamma_1;\gamma_2;\frac{\gamma_1+1}{2};\gamma_2)}( \varpi_1(X-t), \varpi_2(Y-\Omega))\\ &  \times t^{\sigma-1}\Omega^{\eta-1}E_{\sigma,\eta,\Upsilon}^{(\mu_1;\mu_2;\frac{\mu_1+1}{2};\mu_2)}( \varpi_1t, \varpi_2\Omega)d\Omega dt\bigg](p,q)
\\ & =\frac{1}{q^\varrho p^\varkappa}\bigg[\frac{q^\Upsilon-\nu_2^\Upsilon}{q^\Upsilon}\bigg]^{-\gamma_2}\bigg[1-\frac{4\nu_1q^\Upsilon}{p(q^\Upsilon-\nu_2^\Upsilon)}\bigg ]^{-\frac{\gamma_1}{2}}\frac{1}{q^\eta p^\sigma}\bigg[\frac{q^\Upsilon-\nu_2^\Upsilon}{q^\Upsilon}\bigg]^{-\mu_2}\bigg[1-\frac{4\nu_1q^\Upsilon}{p(q^\Upsilon-\nu_2^\Upsilon)}\bigg ]^{-\frac{\mu_1}{2}}
\\ & =\frac{1}{q^{\varrho+\eta} p^{\varkappa+\sigma}}\bigg[\frac{q^\Upsilon-\nu_2^\Upsilon}{q^\Upsilon}\bigg]^{-\gamma_2-\mu_2}\bigg[1-\frac{4\nu_1q^\Upsilon}{p(q^\Upsilon-\nu_2^\Upsilon)}\bigg ]^{\frac{-\gamma_1-\mu_1}{2}}
\\ & =\mathbb{L}_2\big[X^{\varkappa+\sigma-1}Y^{\varrho+\eta-1}E_{\varkappa+\sigma,\varrho+\eta,\Upsilon}^{(\gamma_1+\mu1;\gamma_2+\mu_2;\frac{\gamma_1+\mu_1+1}{2};\gamma_2+\mu2)}( \varpi_1 X, \varpi_2 Y)\big](p,q ).
\end{aligned}
\end{equation*}
taking inverse Laplace on both sides of the above equation, we get the result.
\end{proof}

\begin{theorem} \label{thm20}
The integral operator \eqref{into} satisfies the following semigroup property for any summable function $\psi\in L_1((A,B) \times (C,D))$
\begin{equation}
\begin{aligned}
&{} \bigg({}_{\phantom{1}{A^+,C^+}}\mathbb{I}_{\varkappa,\varrho, \Upsilon,\nu_1,\nu_2}^{(\gamma_1;\gamma_2;\frac{\gamma_1+1}{2};\gamma_2)}{}_{\phantom{1}{A^+,C^+}}\mathbb{I}_{\sigma,\eta, \Upsilon,\nu_1,\nu_2}^{(\mu_1;\mu_2;\frac{\mu_1+1}{2};\mu_2)}\psi\bigg)(X,Y) =\\ & {}_{\phantom{1}{A^+,C^+}}\mathbb{I}_{\varkappa+\sigma,\varrho+\eta,\Upsilon,\nu_1,\nu_2}^{(\gamma_1+\mu_1;\gamma_2+\mu_2;\frac{\gamma_1+\mu_1+1}{2};\gamma_2+\mu_2)}\psi (X,Y).  \label{semi}
\end{aligned}
\end{equation}
in particular,
\begin{equation}
\begin{aligned}
\bigg({}_{\phantom{1}{A^+,C^+}}\mathbb{I}_{\varkappa,\varrho, \Upsilon,\nu_1,\nu_2}^{(\gamma_1;\gamma_2;\frac{\gamma_1+1}{2};\gamma_2)}{}_{\phantom{1}{A^+,C^+}}\mathbb{I}_{\sigma,\eta, \Upsilon,\nu_1,\nu_2}^{(-\gamma_1;-\gamma_2;\frac{-\gamma_1+1}{2};-\gamma_2)}\psi\bigg)(X,Y) {}&={}_{\phantom{1}{A^+,C^+}}\mathbb{I}_{\varkappa+\sigma,\varrho+\eta,\Upsilon}^{(0;0;\frac{1}{2};0)}\psi (X,Y)  
\\&={}_{\phantom{1}Y}I_{C^+}^{\varrho+\eta}{}_{\phantom{1}X}I_{A^+}^{\varkappa+\sigma}\psi (X,Y).
\label{semi2}
\end{aligned}
\end{equation}
\end{theorem}

\begin{proof}
By using \eqref{into}, we get
\begin{equation*}
\begin{aligned}
\bigg(  {}& {}_{\phantom{1}{A^+,C^+}}\mathbb{I}_{\varkappa,\varrho, \Upsilon,\nu_1,\nu_2}^{(\gamma_1;\gamma_2;\frac{\gamma_1+1}{2};\gamma_2)}{}_{\phantom{1}{A^+,C^+}}\mathbb{I}_{\sigma,\eta, \Upsilon,\nu_1,\nu_2}^{(\mu_1;\mu_2;\frac{\mu_1+1}{2};\mu_2)}\psi\bigg)(X,Y)  \\&
=\int_C^Y\int_A^X (X-t)^{\varkappa-1} (Y-\Omega)^{\varrho-1} E_{\varkappa,\varrho,\Upsilon}^{(\gamma_1;\gamma_2;\frac{\gamma_1+1}{2};\gamma_2)}( \varpi_1(X-t), \varpi_2(Y-\Omega))\\ & \times \big({}_{\phantom{1}{A^+,C^+}}I_{\sigma,\eta, \Upsilon,\nu_1,\nu_2}^{(\mu_1;\mu_2;\frac{\mu_1+1}{2};\mu_2)}\psi\big)(t,\Omega)dt d\Omega\\ &
=\int_C^Y\int_A^X\int_C^\Omega\int_A^t (X-t)^{\varkappa-1} (Y-\Omega)^{\varrho-1} E_{\varkappa,\varrho,\Upsilon}^{(\gamma_1;\gamma_2;\frac{\gamma_1+1}{2};\gamma_2)}( \varpi_1(X-t), \varpi_2(Y-\Omega)) \\& 
\times(t-u)^{\sigma-1}(\Omega-v)^{\eta-1} E_{\sigma,\eta, \Upsilon,\nu_1,\nu_2}^{(\mu_1;\mu_2;\frac{\mu_1+1}{2};\mu_2)}( \varpi_1(t-u), \varpi_2(\Omega-v))dudvdtd\Omega \\&
=\int_C^Y\int_A^X\int_0^{Y-v}\int_0^{X-u}(X-k-u)^{\varkappa-1} (Y-l-v)^{\varrho-1} \\& 
\times  E_{\varkappa,\varrho,\Upsilon}^{(\gamma_1;\gamma_2;\frac{\gamma_1+1}{2};\gamma_2)}( \varpi_1(X-k-u), \varpi_2(Y-l-v))
k^{\sigma-1}l^{\eta-1}  E_{\sigma,\eta, \Upsilon,\nu_1,\nu_2}^{(\mu_1;\mu_2;\frac{\mu_1+1}{2};\mu_2)}( \varpi_1k, \varpi_2l)dkdldudv.
\end{aligned}
\end{equation*}
Now, using Theorem \ref{thm19}, we obtain
\begin{equation*}
\begin{aligned}
\bigg(  {}& {}_{\phantom{1}{A^+,C^+}}\mathbb{I}_{\varkappa,\varrho, \Upsilon,\nu_1,\nu_2}^{(\gamma_1;\gamma_2;\frac{\gamma_1+1}{2};\gamma_2)}{}_{\phantom{1}{A^+,C^+}}\mathbb{I}_{\sigma,\eta, \Upsilon,\nu_1,\nu_2}^{(\mu_1;\mu_2;\frac{\mu_1+1}{2};\mu_2)}\psi\bigg)(X,Y) \\&
=\int_C^Y\int_A^X (X-u)^{\gamma_1+\mu_1-1}(Y-v)^{\varrho+\eta-1}E_{\varkappa+\sigma,\varrho+\eta,\Upsilon}^{(\gamma_1+\mu_1;\gamma_2+\mu_2;\frac{\gamma_1+\mu_1+1}{2};\gamma_2+\mu_2)}( \varpi_1(X-u), \varpi_2(Y-v))dudv \\&
={}_{\phantom{1}{A^+,C^+}}\mathbb{I}_{\varkappa+\sigma,\varrho+\eta,\Upsilon,\nu_1,\nu_2}^{(\gamma_1+\mu_1;\gamma_2+\mu_2;\frac{\gamma_1+\mu_1+1}{2};\gamma_2+\mu_2)}\psi (X,Y).
\end{aligned}
\end{equation*}
Whence the result.
\end{proof}

If, $\gamma_1=\gamma_2=0$, \eqref{semi} coincides with the following corollary:,

\begin{corollary}
There holds the relation on $L_1((A,B) \times (C,D))$.
\begin{equation*}
\begin{aligned}
\big({}_{\phantom{1}{A^+,C^+}}I_{\varkappa,\varrho, \Upsilon,\nu_1,\nu_2}^{(0;0;\frac{1}{2};0)}{}_{\phantom{1}{A^+,C^+}}I_{\sigma,\eta, \Upsilon,\nu_1,\nu_2}^{(\mu_1;\mu_2;\frac{\mu_1+1}{2};\mu_2)}\psi\big)(X,Y) {}& =\big({}_{\phantom{1}Y}I_{C^+}^{\varrho}{}_{\phantom{1}X}I_{A^+}^{\varkappa}{}_{\phantom{1}{A^+,C^+}}I_{\sigma,\eta, \Upsilon,\nu_1,\nu_2}^{(\mu_1;\mu_2;\frac{\mu_1+1}{2};\mu_2)}\psi\big)(X,Y)  
\\&  =\big( {}_{\phantom{1}{A^+,C^+}}I_{\varkappa+\sigma,\varrho+\eta, \Upsilon,\nu_1,\nu_2}^{(\mu_1;\mu_2;\frac{1+\mu_1}{2};\mu_2)}\psi\big)(X,Y) .
\end{aligned}
\end{equation*}
\end{corollary}

\begin{corollary}
The following composition relationships hold true for all  $\psi(X,Y) \in L_1((A,B) \times (C,D))$.
\begin{equation*}
\begin{aligned}
\big({}_{\phantom{1}Y}I_{C^+}^{\varrho}{}_{\phantom{1}X}I_{A^+}^{\varkappa}{}_{\phantom{1}{A^+,C^+}}\mathbb{I}_{\sigma,\eta, \Upsilon,\nu_1,\nu_2}^{(\mu_1;\mu_2;\mu_3;\mu_4)}\psi\big)(X,Y) {}& =\big({}_{\phantom{1}{A^+,C^+}}\mathbb{I}_{\varkappa+\sigma,\varrho+\eta, \Upsilon,\nu_1,\nu_2}^{(\mu_1;\mu_2;\mu_3;\mu_4)}\psi\big)(X,Y) 
\\ & =\big({}_{\phantom{1}{A^+,C^+}}\mathbb{I}_{\sigma,\eta, \Upsilon,\nu_1,\nu_2}^{(\mu_1;\mu_2;\mu_3;\mu_4)}{}_{\phantom{1}Y}I_{C^+}^{\varrho}{}_{\phantom{1}X}I_{A^+}^{\varkappa}\psi\big)(X,Y).
\end{aligned}
\end{equation*}

\begin{equation*}
\begin{aligned}
\big({}_{\phantom{1}Y}D_{C^+}^{\varrho}{}_{\phantom{1}X}D_{A^+}^{\varkappa}{}_{\phantom{1}{A^+,C^+}}\mathbb{I}_{\sigma,\eta, \Upsilon,\nu_1,\nu_2}^{(\mu_1;\mu_2;\mu_3;\mu_4)}\psi\big)(X,Y) {}& =\big( {}_{\phantom{1}{A^+,C^+}}\mathbb{I}_{\sigma-\varkappa,\eta-\varrho, \Upsilon,\nu_1,\nu_2}^{(\mu_1;\mu_2;\mu_3;\mu_4)}\psi\big)(X,Y) 
\\ & =\big({}_{\phantom{1}{A^+,C^+}}\mathbb{I}_{\sigma,\eta, \Upsilon,\nu_1,\nu_2}^{(\mu_1;\mu_2;\mu_3;\mu_4)}{}_{\phantom{1}Y}D_{C^+}^{\varrho}{}_{\phantom{1}X}D_{A^+}^{\varkappa}\psi\big)(X,Y).
\end{aligned}
\end{equation*}
\end{corollary}

\begin{theorem}
The integral operator $\big({}_{\phantom{1}{A^+,C^+}}\mathbb{I}_{\varkappa,\varrho, \Upsilon,\nu_1,\nu_2}^{(\gamma_1;\gamma_2;\frac{\gamma_1+1}{2};\gamma_2)}\psi\big)(X,Y)$ has a left inverse operator on the space $L_1((A,B) \times (C,D))$,  namely the operator $\big({}_{\phantom{1}{A^+,C^+}}\mathbb{D}_{\varkappa,\varrho, \Upsilon,\nu_1,\nu_2}^{(\gamma_1;\gamma_2;\frac{\gamma_1+1}{2};\gamma_2)}\psi\big)(X,Y)$ defined as follows.
\begin{equation}
\big({}_{\phantom{1}{A^+,C^+}}\mathbb{D}_{\varkappa,\varrho, \Upsilon,\nu_1,\nu_2}^{(\gamma_1;\gamma_2;\frac{\gamma_1+1}{2};\gamma_2)}\psi\big)(X,Y)={}_{\phantom{1}Y}D_{C^+}^{\varrho+\eta}{}_{\phantom{1}X}D_{A^+}^{\varkappa+\sigma}\big({}_{\phantom{1}{A^+,C^+}}\mathbb{I}_{\varkappa,\varrho, \Upsilon,\nu_1,\nu_2}^{(-\gamma_1;-\gamma_2;\frac{-\gamma_1+1}{2};-\gamma_2)}\psi\big)(X,Y). \label{left}
\end{equation}
\end{theorem}

\begin{proof}
Let $f \in L_1((A,B) \times (C,D)$, and 
\begin{equation}
f(X,Y)=\big({}_{\phantom{1}{A^+,C^+}}\mathbb{I}_{\varkappa,\varrho, \Upsilon,\nu_1,\nu_2}^{(\gamma_1;\gamma_2;\frac{\gamma_1+1}{2};\gamma_2)}\psi\big)(X,Y). \label{left1}
\end{equation}
Applying the operator $\big({}_{\phantom{1}{A^+,C^+}}\mathbb{I}_{\varkappa,\varrho, \Upsilon,\nu_1,\nu_2}^{(-\gamma_1;-\gamma_2;\frac{-\gamma_1+1}{2};-\gamma_2)}\big)$ to both sides of \eqref{left1} and using the semigroup property of Theorem \ref{thm20}, we obtain
\begin{equation*}
\begin{aligned}
\big({}_{\phantom{1}{A^+,C^+}}\mathbb{I}_{\varkappa,\varrho, \Upsilon,\nu_1,\nu_2}^{(-\gamma_1;-\gamma_2;\frac{-\gamma_1+1}{2};-\gamma_2)}f\big)(X,Y) {}&=\big({}_{\phantom{1}{A^+,C^+}}\mathbb{I}_{\varkappa,\varrho, \Upsilon,\nu_1,\nu_2}^{(0;0;\frac{1}{2};0)}\psi\big)(X,Y) 
\\&={}_{\phantom{1}Y}I_{C^+}^{\varrho+\eta}{}_{\phantom{1}X}I_{A^+}^{\varkappa+\sigma}\psi (X,Y).
\end{aligned}
\end{equation*}
Applying double fractional derivate operator to the above equation, we get
\begin{equation*}
{}_{\phantom{1}Y}D_{C^+}^{\varrho+\eta}{}_{\phantom{1}X}D_{A^+}^{\varkappa+\sigma}\big({}_{\phantom{1}{A^+,C^+}}\mathbb{I}_{\varkappa,\varrho, \Upsilon,\nu_1,\nu_2}^{(-\gamma_1;-\gamma_2;\frac{-\gamma_1+1}{2};-\gamma_2)}f\big)(X,Y)=\psi(X,Y),
\end{equation*}
i.e.,
\begin{equation*}
\big({}_{\phantom{1}{A^+,C^+}}\mathbb{D}_{\varkappa,\varrho, \Upsilon,\nu_1,\nu_2}^{(\gamma_1;\gamma_2;\frac{\gamma_1+1}{2};\gamma_2)}f\big)(X,Y)=\psi(X,Y),
\end{equation*}
which means that ${}_{\phantom{1}{A^+,C^+}}\mathbb{D}_{\varkappa,\varrho, \Upsilon,\nu_1,\nu_2}^{(\gamma_1;\gamma_2;\frac{\gamma_1+1}{2};\gamma_2)}$ is the left inverse of ${}_{\phantom{1}{A^+,C^+}}\mathbb{I}_{\varkappa,\varrho, \Upsilon,\nu_1,\nu_2}^{(\gamma_1;\gamma_2;\frac{\gamma_1+1}{2};\gamma_2)}$ on the function space $L_1((A,B) \times (C,D)$.
\end{proof}

\section{Conclusions}
We introduced a new technique for constructing bivariate biorthogonal polynomial families and, using this method, defined the 2D Hermite-Konhauser polynomials. We also defined the bivariate H–K Mittag-Leffler functions corresponding to this new class of bivariate orthogonal polynomials and presented several important properties of both. \\
It is worth noting that this technique allows for constructing additional families of bivariate biorthogonal polynomials.For instance, a new class of bivariate biorthogonal polynomials can be constructed using Jacobi and Konhauser polynomials. These may be termed 2D Jacobi–Konhauser polynomials, with the following explicit form
\begin{equation}
{}_{\phantom{1}\Upsilon}P_n^{(\varkappa,\varrho)}(X,Y)=\frac{\Gamma(1+\varkappa+n)}{n!}\sum_{s=0}^n\sum_{r=0}^{n-s}\frac{(-n)_{s+r}(1+\varkappa+\varrho+n)_s}{s!r!\Gamma(1+\varkappa+s)\Gamma(\varrho+1+\Upsilon r)}\left(\frac{1-X}{2}\right)^sY^{\Upsilon r} \label{jac}
\end{equation}
where $\varkappa>-1$ , $\varrho>-1$ and $\Upsilon=1,2\dots$. \\
The corresponding bivariate J - K Mittag-Leffler function is defined by
\begin{equation}
E_{\varkappa,\varrho,\Upsilon}^{(\gamma_1;\gamma_2)}(X,Y) =\sum_{s=0}^\infty\sum_{r=0}^\infty\frac{(\gamma_1)_{r+s}(\gamma_2)_{s}X^sY^{\Upsilon r}}{r!s!\Gamma(\varkappa+s)\Gamma(\varrho+\Upsilon r)},
\end{equation}
where $ (\varkappa, \varrho, \gamma_1,\gamma_2 \in\mathbb{C}, Re(\varkappa), Re(\varrho), Re(\Upsilon)>0)$, \\and hence there is the following relation between them;
\begin{equation*}
{}_{\phantom{1}\Upsilon}P_n^{(\varkappa,\varrho)}(X,Y)=\frac{\Gamma(1+\varkappa+n)}{n!}E_{\varkappa+1,\varrho+1,\Upsilon}^{(-n;1+\varkappa+\varrho+n)}(\frac{1-X}{2},Y). \label{rel}
\end{equation*}
A secondary aim of this paper was to establish a new framework for fractional calculus based on the 2D Hermite-Konhauser polynomials. To achieve this, we introduced two additional parameters to the 2D Hermite-Konhauser polynomials, defining a new bivariate polynomial family called the modified 2D Hermite-Konhauser polynomials. We also defined the corresponding modified bivariate Mittag-Leffler functions. Using the modified bivariate Mittag-Leffler functions as kernels, we constructed a new model of fractional calculus that possesses the semigroup property.\\

\section*{Declarations}

\textbf{Conflict of interest} All authors declare that they have no conflict of interest regarding this manuscript.\\
\textbf{Informed Constent} Not applicable. \\
\textbf{Data Available} No data used.

\end{document}